\documentclass[a4paper,12pt]{article}
\usepackage{mathtools}
\usepackage{amsmath}
\usepackage{amssymb}
\usepackage{amsthm}
\usepackage[left=2cm,right=2cm,top=3.5cm,bottom=3.5cm]{geometry}
\usepackage{graphicx}
\usepackage{tikz}
\usepackage{tikz-network}
\usepackage{bbm}
\usepackage{wasysym}
\usepackage{eucal}
\usepackage{mathrsfs}
\usepackage{float}
\usepackage[skip=0pt]{caption}
\usepackage{enumerate}
\usepackage{cite}
\usepackage{hyperref}

\theoremstyle{plain}
\newtheorem{thm}{Theorem}
\newtheorem{lem}[thm]{Lemma}
\newtheorem{cor}[thm]{Corollary}
\newtheorem{pro}[thm]{Proposition}

\theoremstyle{definition}
\newtheorem{defn}[thm]{Definition}
\newtheorem{rem}[thm]{Remark}
\newtheorem{ex}[thm]{Example}

\def\al{\textit{mm-alternating }}
\def\B{\mathscr{B}}
\def\Bt{\mathscr{B}({\theta})}
\def\Bi{\mathscr{B}({\infty})}
\def\i{\textsl{i}}
\def\M{\mathcal{M}}

\begin{document}
	\title{On the singularity and the inverse of 3-colored digraphs}
	\author{
		Md Isheteyak Zaffer\footnote{Department of Mathematical Sciences, Tezpur University, Tezpur, Assam-784028, India; (isheteyak.zaffer@gmail.com).} }
	\date{}
	\maketitle
	\begin{abstract}
		This article investigates the singularity and invertibility of simple connected 3-colored digraphs, with particular emphasis on the classes of unicyclic and bicyclic digraphs. A digraph $G$ is said to be non-singular (resp. singular) if its adjacency matrix $A(G)$ is non-singular (resp. singular). The main objective of this article is to characterize the non-singular 3-colored unicyclic and bicyclic digraphs. For a non-singular digraph $G$, if the diagonal entries of $A(G)^{-1}$ are all zero, then $A(G)^{-1}$ can be realized as the adjacency matrix of a weighted digraph with complex weights. We characterize all non-singular 3-colored unicyclic and bicyclic digraphs for which this condition holds. Furthermore, we investigate the invertibility of these digraphs and characterize those whose inverses are also 3-colored digraphs. In particular, our characterization of the unicyclic case extends earlier results by removing the assumption that the digraph possesses a unique perfect matching.
	\end{abstract}

	
	
	\section{Introduction}
	
	All graphs considered in this study are connected and simple, with edge weights from the set $\{\pm1,\pm\i\}$, where $\i=\sqrt{-1}$.
	
	Let $G=(V(G),E(G))$ be a graph with vertex set $V(G)$ and edge set $E(G)$. A connected graph $G$ is called \textit{$k$-cyclic} if its order is $n$ and its size is $n+k-1$. In particular, a connected graph is called \textbf{unicyclic} if both its order and size are equal, and \textbf{bicyclic} if its order is $n$ and its size is $n+1$. It is clear that a unicyclic graph has exactly one cycle, whereas a bicyclic graph contains at least two cycles. If the cycles in a bicyclic graph $B$ are edge-disjoint, then $B$ is called an $\infty$-type graph; otherwise, it is referred to as a $\theta$-type graph. Thus, in an $\infty$-type bicyclic graph, the two cycles share at most one vertex, whereas in a $\theta$-type bicyclic graph, the cycles share at least two vertices.
	
	Let $G$ be a simple undirected graph and let $\vec{G}$ be a digraph whose underlying graph is $G$. We use $[i,j]$ to denote an edge between vertices $i$ and $j$ in $G$, and $(i,j)$ to denote an arc from $i$ to $j$ in $\vec{G}$. In what follows, we use $G$ to denote both the digraph and its underlying graph whenever the distinction is clear from the context. Let $G$ be a digraph on $n$ vertices, with weight $w_{ij}$ assigned to the arc $(i,j)$, and let $\overline{w}_{ij}$ denote the complex conjugate of $w_{ij}$. The adjacency matrix $A(G)=[a_{ij}]$ is an $n\times n$ matrix defined by
	\[
	a_{ij}=
	\left\{
	\begin{array}{ll}
		w_{ij} & \text{if }(i,j)\in E(G),\\
		\overline{w}_{ij} & \text{if }(j,i)\in E(G),\\
		0 & \text{otherwise}.
	\end{array}
	\right.
	\]
	
	From the definition of $A(G)$, it is clear that, as far as the adjacency matrix is concerned, the weight of an arc can be assumed to belong to $\{-1,1,\i\}$. Indeed, an arc $(i,j)$ with weight $-\i$ can be replaced by the arc $(j,i)$ with weight $\i$ without changing the adjacency matrix. Moreover, the direction of an arc with weight $\pm1$ is immaterial in the study of the adjacency matrix. Therefore, we treat arcs with weights $\pm1$ as undirected edges and always assume that a directed arc has weight $\i$. We label edges with weights $1$ and $-1$ as red and blue, respectively, and arcs with weight $\i$ as green. A digraph with these three types of edges and arcs is called a \textit{3-colored digraph}. The concept of 3-colored digraphs was introduced by Bapat et al. \cite{DK}. We denote the classes of unicyclic and bicyclic 3-colored digraphs by $\mathscr{U}$ and $\B$, respectively. Moreover, the subclasses of $\infty$-type and $\theta$-type bicyclic digraphs are denoted by $\B(\infty)$ and $\B(\theta)$, respectively.
	
	Let $G$ be a (di)graph. The weight of a cycle $\Gamma=[i_1,\ldots,i_n,i_1]$ in $G$, denoted by $w(\Gamma)$, is the product
	\(
	w(i_1,i_2)\times\cdots\times w(i_{n-1},i_n)\times w(i_n,i_1),
	\)
	where $w(i,j)$ denotes the weight associated with the arc $(i,j)$. The weight of a path is defined similarly. An \textit{elementary subgraph} $H$ of $G$ is a subgraph in which each component is either a cycle or an independent edge. The subgraph $H$ is called a \textit{spanning elementary subgraph} if $V(H)=V(G)$. A component of an elementary subgraph is called \textit{non-singular} if it is a cycle whose weight is different from $1$; all other components are called \textit{singular}. A \textit{perfect matching} $\M$ of a graph $G$ is a spanning elementary subgraph in which every component is an independent edge. Here, the terms subgraph, path, matching, and component are understood with respect to the underlying graph of the digraph.
	
	A spanning elementary subgraph is called a \textit{contributing spanning elementary subgraph} if it contains no cycle of weight $\pm\i$. A graph (digraph) is said to be \textit{non-singular} (resp. \textit{singular}) if its adjacency matrix is non-singular (resp. singular). Singularity is a fundamental property of a digraph and plays an important role in the study of its spectral properties and invertibility. The following lemma from \cite{DK1} provides a formula for calculating the determinant of the adjacency matrix of a 3-colored digraph.
	
	\begin{lem}(\cite{DK1}, Lemma 3)\label{adjdet}
		Let $G$ be a 3-colored digraph on $n$ vertices. Then
		\[
		\det A(G)=\sum_{H\in \mathscr{C}}(-1)^{n-|S_H|}2^{|C_H|},
		\]
		where $\mathscr{C}$ is the set of all contributing spanning elementary subgraphs of $G$, and $|S_H|$ and $|C_H|$ denote the number of singular components and the number of cycles in $H$, respectively.
	\end{lem}
	
	The degree $d_i$ of a vertex $i$ in a graph $G$ is defined as the number of edges incident to $i$, including both directed and undirected edges. Let $D(G)$ denote the diagonal matrix whose diagonal entries are the degrees $d_i$ for each vertex $i\in V(G)$. The Laplacian matrix $L(G)$ of $G$ is then defined by $L(G)=D(G)-A(G)$, where $A(G)$ is the adjacency matrix of $G$. Bapat, Kalita, and Pati in \cite{DK} characterized 3-colored digraphs for which the Laplacian matrix is non-singular. Kalita and Pati further investigated the reciprocal eigenvalue properties of unicyclic 3-colored digraphs in \cite{DK1}. Kalita and Sarma \cite{DK2} studied the invertibility of unicyclic 3-colored digraphs with a unique perfect matching and characterized those whose inverses are also unicyclic. Extending this, Kalita and Sarma in \cite{DK3} characterized unicyclic 3-colored digraphs with a unique perfect matching whose inverses are bicyclic. Graphs with edge weights of unit modulus are referred to as complex unit gain graphs; see \cite{NR}. Recent contributions to the study of complex unit gain graphs, particularly in relation to their adjacency matrices and eigenvalues, can be found in \cite{LL, FQ, R} and the references therein.
	
	The non-singularity of (di)graphs is pivotal in the analysis of their spectral properties and is closely related to their nullity. This motivates the following questions: \textit{Does there exist a non-singular 3-colored digraph?} If so, \textit{can such digraphs be characterized?} In this paper, we address these questions for digraphs belonging to the classes $\mathscr{U}$ and $\mathscr{B}$. If $A(G)$ is non-singular, then its inverse $A(G)^{-1}$ is well-defined. When all diagonal entries of $A(G)^{-1}$ are zero, it can be regarded as the adjacency matrix of another weighted graph with complex weights. This weighted (di)graph with the adjacency matrix $A(G)^{-1}$ is called the \textbf{inverse} of $G$. We identify the non-singular digraphs in $\mathscr{U}$ and $\mathscr{B}$ for which the diagonal entries of the inverse of the adjacency matrix are zero. The concept of a graph inverse was first introduced by Harary and Minc in \cite{HM}, who defined a graph $G$ to be invertible if $A(G)^{-1}$ consists entirely of $0$ and $1$ entries, with the resulting graph being called the inverse of $G$. They also showed that $K_2$ is the only connected invertible graph under this definition. Given the restrictive nature of this definition, several other notions of graph inverses have subsequently been explored; see \cite{god, DK4}. The study of graph invertibility has attracted considerable attention; see \cite{DK2, DK4, god, PP1, AK} and the references therein. While research on graph inverses has largely been confined to undirected graphs, with the notable exceptions of \cite{DK2, DK3}, we extend the investigation to general 3-colored unicyclic and bicyclic digraphs. In particular, the studies in \cite{DK2} and \cite{DK3} are restricted to unicyclic 3-colored digraphs possessing a unique perfect matching, whereas no such restriction is imposed in the present work. For instance, we establish a complete characterization of the invertible digraphs $U\in\mathscr{U}$ whose inverses are 3-colored: this occurs precisely when $U$ has a unique perfect matching and the number of pegs on its unique cycle is different from two. Thus, our characterization extends the earlier result of Kalita and Sarma to the entire class of invertible unicyclic 3-colored digraphs.
	
	It follows from Lemma~\ref{adjdet} that, when all cycles of a 3-colored digraph have weights $\pm1$, the determinant of its adjacency matrix is governed by the same type of contributing subgraphs as in the case of ordinary signed graphs (undirected graphs with edge weights $\pm1$). Thus, to distinguish the present study from that of ordinary signed graphs, we assume that at least one cycle in the digraph has weight $\pm\i$. Since a tree contains no cycles, the study of its adjacency matrix in the present setting reduces to that of an ordinary simple tree. Consequently, we focus on digraphs containing at least one cycle.
	
	This article is structured as follows. Section~\ref{u} examines the singularity of 3-colored digraphs in $\mathscr{U}$ and $\mathscr{B}$ and provides characterizations of non-singular digraphs in these classes. Section~\ref{3} identifies the non-singular digraphs in $\mathscr{U}\cup\mathscr{B}$ for which the diagonal entries of the inverse adjacency matrix are zero. Section~\ref{4} develops the concept of graph inverses and, using the results established in Sections~\ref{u} and \ref{3}, determines all digraphs in $\mathscr{U}\cup\mathscr{B}$ whose inverses are also 3-colored digraphs. In addition, this section identifies all unimodular digraphs in both $\mathscr{U}$ and $\mathscr{B}$, that is, digraphs whose adjacency matrices have determinant $\pm1$.
	
	\section{Non-singular 3-colored digraphs}\label{u}
	
	Let $G$ be a (di)graph and $H$ be a subgraph of $G$. By $G-H$ we mean the induced subgraph of $G$ on the vertex set $V(G)-V(H)$. We denote the length of a cycle $\Gamma$ as $|\Gamma|$, and the length of a path $P$ as $|P|$. A path from vertex $i$ to $j$ is denoted as $i\leadsto j$ and an $i\leadsto j$ path that contains a $u\leadsto v$ path $P$ is denoted as $[i,\dots,u,P,v,\dots,j]$. Let $\Gamma_1,\dots, \Gamma_k$ be cycles in a graph $G$. Then, $G$ has a spanning elementary subgraph with cycles $\Gamma_{i_1},\dots, \Gamma_{i_j}$ as components if and only if $G-(\Gamma_{i_1}\cup\dots\cup\Gamma_{i_j})$ has a perfect matching, and cycles $\Gamma_{i_k}$ and $\Gamma_{i_l}$ are mutually vertex-disjoint for all $i_k,i_l$. The following remark is straightforward.
	
	\begin{rem}\label{detn}
		Let $G$ be a 3-colored digraph on $n$ vertices. Let $m_0$ be the number of perfect matchings of $G$, $m_{1_i}$ be the number of perfect matchings of $G-\Gamma_i$ where $w(\Gamma_i)\neq\pm\i$, $m_{2_{ij}}$ be the number of perfect matchings of $G-(\Gamma_i\cup\Gamma_j)$ where $w(\Gamma_i)\neq \pm\i\neq w(\Gamma_j)$ and are vertex-disjoint, etc. Then, by Lemma~\ref{adjdet},
		\begin{equation*}
			\begin{split}
				\det A(G)&=m_0 (-1)^{\frac{n}{2}}+2 \sum_{i}m_{1_i}(-1)^{n-\frac{n-|\Gamma_i|}{2}-S_i}+4\sum_{i\neq j}m_{2_{ij}}(-1)^{n-\frac{n-(|\Gamma_i|+|\Gamma_j|)}{2}-(S_i+S_j)}+\cdots\\
				&=m_0 (-1)^{\frac{n}{2}}+2\sum_{i}m_{1_i}(-1)^{\frac{n+|\Gamma_i|}{2}-S_i}+4\sum_{i\neq j}m_{2_{ij}} (-1)^{\frac{n+(|\Gamma_i|+|\Gamma_j|)}{2}-(S_i+S_j)}+\cdots,
			\end{split}
		\end{equation*}
		where $S_i=0$ if $w(\Gamma_i)=-1$ and $S_i=1$ if $w(\Gamma_i)=1$. The expression can be written as $$\det A(G)= m_0 (-1)^{\frac{n}{2}}-2\sum_{i}m_{1_i}w(\Gamma_i)(-1)^{\frac{n+|\Gamma_i|}{2}}+4\sum_{i\neq j}m_{2_{ij}}w(\Gamma_i)w(\Gamma_j) (-1)^{\frac{n+(|\Gamma_i|+|\Gamma_j|)}{2}}+\cdots.$$
	\end{rem}

	The following corollaries are the immediate consequence of Remark~\ref{detn}.
	\begin{cor}\label{upm}
		Let $G$ be a 3-colored digraph. If $G$ has a unique perfect matching, then $G$ is non-singular.
	\end{cor}
	\begin{proof}
		This follows from Remark~\ref{detn}, as $\det A(G)\equiv m_0\pmod2$ and $m_0 = 1$ is the number of perfect matching in $G$.
	\end{proof}
	
	\begin{cor}\label{npm}
		Let $G$ be a 3-colored digraph. If every cycle in $G$ has weight $\pm \i$, then $\det A(G)=\pm m_0$, where $m_0$ denotes the number of perfect matchings of $G$.
	\end{cor}
	\begin{proof}
		Since the weight of every cycle in $G$ is $\pm\i$, $G$ does not have any contributing spanning elementary subgraph other than the perfect matching(s). Hence, $\det A(G)=m_0\times (-1)^{\frac{|V(G)|}{2}}$, by Remark~\ref{detn}.
	\end{proof}

	\begin{lem}\label{evennomatch}
		Let $G$ be a digraph without any perfect matching such that all cycles in $G$ are even. Then $\det A(G)=0$.
	\end{lem}
	\begin{proof}
		Let $H$ be a spanning elementary subgraph of $G$. Since $G$ does not contain perfect matching, $H$ must necessarily include at least one cycle. As every even cycle admits a perfect matching, combining the perfect matchings of the cycle components of $H$ with the independent edges of $H$ would produce a perfect matching of $G$, which contradicts the fact that $G$ does not have a perfect matching. Therefore, $G$ cannot contain any spanning elementary subgraph, implying that $\det A(G) = 0$, by Remark~\ref{detn}.
	\end{proof}
	
	The following proposition provides a useful characterization of graphs having more than one perfect matching. Although the result is standard, we include a proof for completeness.
	
	\begin{pro}\label{mperf}
		Let $G$ be a graph. Then $G$ has more than one perfect matchings if and only if $G$ contains at least one even cycle $\Gamma$ such that $G-\Gamma$ also has a perfect matching.
	\end{pro}
	\begin{proof}
		First, assume that $G$ has more than one perfect matching. We begin by showing that $G$ contains an even cycle. Let $\M$ and $\M'$ be two distinct perfect matchings of $G$. It follows that the symmetric difference $\M \Delta \M'$ is non-empty. If $[u,v]$ is a pendant edge in $G$, then $[u,v] \in \M \cap \M'$, implying $[u,v] \notin \M \Delta \M'$. Consequently, if $[u,v] \in \M \Delta \M'$, then $deg(u) \geq 2$ and $deg(v) \geq 2$. Without loss of generality, assume that $\M \setminus\M'$ is non-empty, and let $[u,v] \in \M \setminus \M'$. Then, there exist vertices $u_1$ and $v_1$ such that $[u,u_1],[v,v_1] \in \M' \setminus \M$.
		
		If $[u_1,v_1]$ is an edge, then $\Gamma = [u,u_1,v_1,v]$ forms an even cycle in $G$. Otherwise, we find vertices $u_2$ and $v_2$ such that $[u_1,u_2],[v_1,v_2] \in \M \setminus \M'$. Now, if $[u_1,v_2]\in E(G)$, then $\Gamma=[u,u_1,u_2,v_2,v_1,v,v]$ is an even cycle, otherwise we find vertices $u_3, v_3$ such that $[u_2,u_3],[v_2,v_3]\in \M'\setminus\M$. Since $G$ is finite, this process continues until we find vertices $u_j$ and $v_j$ where $[u_j,v_j] \in E(G)$. Hence, $\Gamma = [u,u_1,\dots,u_j,v_j,\dots,v_1,v]$ is the desired even cycle in $G$. 
		
		We now show that \(G-\Gamma\) has a perfect matching. If \(j\) is even, then the edges
		\[
		[u,v],\ [u_1,u_2],\ldots,[u_{j-1},u_j],
		[v_1,v_2],\ldots,[v_{j-1},v_j]
		\]
		belong to \(\mathcal M\) and form a perfect matching of \(\Gamma\). Hence, \(\mathcal M\setminus E(\Gamma)\) is a perfect matching of \(G-\Gamma\). On the other hand, if \(j\) is odd, then the edges
		\[
		[u,u_1],\ [u_2,u_3],\ldots,[u_{j-1},u_j],
		[v,v_1],\ [v_2,v_3],\ldots,[v_{j-1},v_j]
		\]
		belong to \(\mathcal M'\) and form a perfect matching of \(\Gamma\). Therefore, \(\mathcal M'\setminus E(\Gamma)\) is a perfect matching of \(G-\Gamma\). This establishes the forward implication.
		
		Conversely, assume that $G$ contains an even cycle $\Gamma = [u_1,\dots,u_{n},u_1]$ such that $G - \Gamma$ has a perfect matching $\M_0$. Then, $\M_0 \cup \{[u_1,u_2],\dots,[u_{n-1},u_n]\}$ and $\M_0 \cup \{[u_1,u_n],\dots,[u_3,u_2]\}$ are two distinct perfect matchings of $G$.
	\end{proof}
	
	\begin{defn}\label{ind}
		Let $G$ be a graph. We call a cycle $\Gamma$ \textbf{independent} in $G$ if $G-\Gamma$ has a perfect matching.
	\end{defn}
	
	\begin{rem}\label{rem1}
		From Proposition~\ref{mperf}, it follows that an independent cycle in a graph with more than one perfect matching is even, while in a graph with no perfect matching, it is odd. Furthermore, a graph with a unique perfect matching does not contain any independent cycle.
	\end{rem}
	
	\subsection{Non-singular digraphs in $\mathscr{U}$}
	
	The following theorem provides complete characterization of non-singular digraphs in $\mathscr{U}$.
	
	\begin{thm}\label{unons}
		Let $U\in\mathscr{U}$. Then $U$ is non-singular if and only if $U$ has a perfect matching.
	\end{thm}
	\begin{proof}
		Let $\Gamma$ be the unique cycle in $U$. By assumption, $w(\Gamma) = \pm\i$, hence, by Corollary~\ref{npm}, $\det A(B) = m_0 \times (-1)^{\frac{|V(U)|}{2}}$, where $m_0$ is the number of perfect matchings of $U$.		
	\end{proof}
	
	\subsection{Non-singular digraphs in $\B$}\label{2}
	
	In this subsection, we identify non-singular bicyclic 3-colored digraphs. Before beginning our study, we note the following property of a $\theta-$type bicyclic graphs.
	
	\begin{rem}\label{paths}
		In a $\theta$-type bicyclic graph $B$, there exist two vertices $i$ and $j$, and three $i \leadsto j$ paths such that these paths have no edge in common. Throughout the article we denote these paths by $\mathcal{P}_1$, $\mathcal{P}_2$, and $\mathcal{P}_3$. $B$ has three cycles given by $\Gamma_1 = \mathcal{P}_1 \cup \mathcal{P}_2$, $\Gamma_2 = \mathcal{P}_1 \cup \mathcal{P}_3$, and $\Gamma_3 = \mathcal{P}_2 \cup \mathcal{P}_3$. See Figure~\ref{pdiag}.
	\end{rem}
	\begin{figure}[H]
		\begin{center}
			\begin{tikzpicture}[scale=0.8]
				\SetVertexStyle[MinSize=0.1,FillColor=black]
				\Vertex[label=$i$,position=left]{1}
				\Vertex[x=0.7]{2}
				\Vertex[x=1.4]{3}
				\Vertex[x=2.1,label=$j$,position=above]{4}
				\Vertex[x=0.3,y=-1]{5}
				\Vertex[x=1,y=-1,label=$\mathcal{P}_{3}$,position=-90]{6}
				\Vertex[x=1.7,y=-1]{7}
				\Vertex[x=0.7,y=1]{8}
				\Vertex[x=1.4,y=1]{9}
				\Vertex[x=3]{10}
				\Vertex[x=2.5,y=-1]{11}
				
				\Edge[](1)(2)
				\Edge[label=$\mathcal{P}_{2}$,position=above](2)(3)
				\Edge[](3)(4)
				\Edge[](1)(5)
				\Edge[](5)(6)
				\Edge[](6)(7)
				\Edge[](7)(4)
				\Edge[](1)(8)
				\Edge[label=$\mathcal{P}_{1}$,position=above](8)(9)
				\Edge[](9)(4)
				\Edge(4)(10)
				\Edge(7)(11)
			\end{tikzpicture}
		\end{center}\caption{Paths described in Remark~\ref{paths}.}\label{pdiag}
	\end{figure}
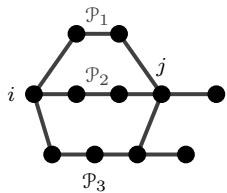
	
	From Remark~\ref{paths}, it follows that if one of the paths, say $\mathcal{P}_1$, is even (resp. odd) while the other two are odd (resp. even), then $\Gamma_1$ and $\Gamma_2$ are odd, whereas $\Gamma_3$ is even. Similarly, if all three paths are even (or all are odd), then all cycles in $B$ are even. In particular, $B$ contains either exactly one even cycle or all its cycles are even.
	
	Furthermore, note that $w(\Gamma_1) = \pm w(\mathcal{P}_1)w(\mathcal{P}_2)$, $w(\Gamma_2) = \pm w(\mathcal{P}_1)w(\mathcal{P}_3)$, and $w(\Gamma_3) = \pm w(\mathcal{P}_2)w(\mathcal{P}_3)$. Therefore, the weights of all three cycles are either $\pm 1$, or exactly one of them has weight $\pm 1$. Henceforth, exactly one cycle in a digraph $B\in\B$ has weight $\pm1$.
	
	\begin{rem}\label{det}
		Let $B \in \B$ be a digraph having $n$ vertices. Since a spanning elementary subgraph of $B$ contains at most two cycles and the weight of at least one of which is $\pm \i$, it follows that $m_{k_{i_1 \dots}} = 0$ for $k \geq 2$ in Remark~\ref{detn}. Therefore, the expression in Remark~\ref{detn} reduces to
		$$\det A(B)=m_0\times(-1)^{\frac{n}{2}}-2\sum_{i}m_{1_i}\mathfrak{R}(w(\Gamma_i))\times(-1)^{\frac{n+|\Gamma_i|}{2}},$$ where $\mathfrak{R}(w(\Gamma_i))$ is the real part of $w(\Gamma_i)$.
	\end{rem}

	\begin{lem}\label{noper}
		Let $B\in\B$ be a digraph without a perfect matching. Then $B$ is non-singular if and only if $B$ has an independent cycle of weight $\pm1$.
	\end{lem}
	\begin{proof}
		First, assume that $B$ has an independent cycle $\Gamma$ with $w(\Gamma) = \pm 1$. By Remark~\ref{rem1} $\Gamma$ is an odd cycle. Thus, $B$ contains a contributing spanning elementary subgraph that has $\Gamma$ as a component. Furthermore, as $B$ does not possess a perfect matching and the weights of all the other cycles are $\pm \i$, $B$ does not possess any other contributing spanning elementary subgraph. Hence, by Remark~\ref{det}, $\det A(B) = -2w(\Gamma) \times m(-1)^{\frac{n + |\Gamma|}{2}}$, where $m$ is the number of perfect matchings of $B - \Gamma$.
		
		The converse follows from Corollary~\ref{npm} as $m_0=0$.
	\end{proof}

	Next, we investigate the non-singularity of digraphs in $\B$ that have more than one perfect matchings.
	
	\begin{lem}\label{123}
		Let $B\in\B$ be such that $B$ has a perfect matching. If $B$ has exactly one independent cycle, then $B$ has exactly three spanning elementary subgraphs: two perfect matchings and a subgraph containing $\Gamma$.
	\end{lem}
	\begin{proof}
		Let $\Gamma$ be the independent cycle in $B$. By Remark~\ref{rem1}, $\Gamma$ is even and has two perfect matchings, $M_1$ and $M_2$. Consequently, $B - \Gamma$ must have exactly one perfect matching, $\M_0$. Suppose, to the contrary, that this is not the case. Then, by Proposition~\ref{mperf}, $B - \Gamma$ would contain an independent cycle, say $\Gamma'$. Let $\M_1$ be a perfect matching of $B - (\Gamma \cup \Gamma')$. The union of $\M_1$ with a perfect matching of $\Gamma$ would yield a perfect matching of $B - \Gamma'$, contradicting the assumption that $B$ has only one independent cycle.
		
		Clearly, $\M_0 \cup M_1$ and $\M_0 \cup M_2$ are two distinct perfect matchings of $B$. Additionally, $\M_0 \cup \Gamma$ is a spanning elementary subgraph in $B$.
	\end{proof}

	\begin{lem}\label{1ind}
		Let $B\in\B$ such that $B$ has a perfect matching. If $B$ has exactly one independent cycle $\Gamma$, then $\det A(B)$ is given by $$\det A(B)=\left\{\begin{array}{ll}
			2(-1)^{\frac{n}{2}}\left[1-w(\Gamma)(-1)^{\frac{|\Gamma|}{2}}\right] & \text{ if } w(\Gamma)=\pm1\\
			2(-1)^{\frac{n}{2}} & \text{ if } w(\Gamma)=\pm\i.
		\end{array}\right.$$
	\end{lem}
	\begin{proof}
		The result then follows from Remark~\ref{det}.
	\end{proof}
	
	In the following, we examine the non-singularity of digraphs in $\B$ that contain a perfect matching and possess more than one independent cycle. Due to the structural differences between $\infty$-type and $\theta$-type digraphs, we treat these cases separately.
	
	\subsubsection{$\infty-$Type bicyclic digraph}
	
	In this section, we identify non-singular digraph in $\B(\infty)$.
	
	\begin{lem}\label{lem2}
		Let $B \in \B(\infty)$ be a digraph that has a perfect matching. Let $\Gamma$ and $\Gamma'$ be cycles in $B$. If $V(\Gamma) \cap V(\Gamma')$ is non-empty, then $B$ contains at most one independent cycle.
	\end{lem}
	\begin{proof}
		If $B$ has a unique perfect matching, then by Remark~\ref{rem1}, $B$ does not contain any independent cycle. Therefore, assume that $B$ has more than one perfect matching. By Proposition~\ref{mperf}, $B$ contains an independent cycle say, $\Gamma$.
		
		For contradiction, suppose that $B$ also contains another independent cycle $\Gamma'$. By Remark~\ref{rem1}, $\Gamma'$ must be even. Let $\Gamma = [1,2,\dots,k]$ and $\Gamma' = [1,2',\dots,l']$, where both $k$ and $l$ are even, and $1 \in V(\Gamma) \cap V(\Gamma')$. Note that both $B - \Gamma$ and $B - \Gamma'$ are forests, and by Proposition~\ref{mperf}, each of these forests has a unique perfect matching, which we denote by $\M$ and $\M'$, respectively.
		
		Since $V(\Gamma') - {1}$ contains an odd number of vertices, there must exist a vertex $m' \in V(\Gamma')$ such that $[m', n'] \in \M$ with $n' \notin V(\Gamma) \cup V(\Gamma')$. Moreover, because $B - \Gamma'$ has a perfect matching, there exists a vertex $n_1$ such that $[n', n_1] \in \M'$. Observe that $n_1 \notin V(\Gamma) \cup V(\Gamma')$; otherwise, we could construct a cycle $[1 \leadsto m', n', n_1 \leadsto 1]$ in $B$, which is distinct from both $\Gamma$ and $\Gamma'$.
		
		Since $n'$ is already matched to $m'$ in $\M$, there exists a vertex $n_2 \notin V(\Gamma) \cup V(\Gamma')$ such that $[n_1, n_2] \in \M$. By continuing this argument alternately with $\M$ and $\M'$, we generate an infinite path $[m', n', n_1, n_2, \dots]$ in $B$, which is impossible. Hence, $\Gamma'$ cannot be independent in $B$.
	\end{proof}

	\begin{lem}\label{9span}
		Let $B\in\B(\infty)$ be a digraph that has a perfect matching. Let $\Gamma$ and $\Gamma'$ be cycles in $B$. If both $\Gamma$ and $\Gamma'$ are independent in $B$, then $B$ has nine spanning elementary subgraphs: four perfect matchings, two spanning elementary subgraphs containing only $\Gamma$ as cycle component, two containing only $\Gamma'$, and one containing both $\Gamma$ and $\Gamma'$. 
	\end{lem}
	\begin{proof}
		Since both $\Gamma$ and $\Gamma'$ are independent, they are vertex-disjoint by Lemma~\ref{lem2} and even by Remark~\ref{rem1}. Let $M_1, M_2$ be the perfect matchings of $\Gamma$, and $M'_1, M'_2$ those of $\Gamma'$.
		
		Since the cycles are vertex-disjoint, $\Gamma$ is independent in $B - \Gamma'$ and $\Gamma'$ is independent in $B - \Gamma$. Hence, by Proposition~\ref{mperf}, there exist perfect matchings $\M_1, \M_2$ of $B - \Gamma$ and $\M'_1, \M'_2$ of $B - \Gamma'$. Consequently, $\M = \M_1 \cap \M'_1$ is a perfect matching of $B - (\Gamma \cup \Gamma')$.
		
		It follows that $\M \cup M_i \cup M'_j$ $(i,j \in \{1,2\})$ are four distinct perfect matchings of $B$. Moreover, $\M_1 \cup \Gamma$, $\M_2 \cup \Gamma$, $\M'_1 \cup \Gamma'$, $\M'_2 \cup \Gamma'$, and $\M \cup \Gamma \cup \Gamma'$ are all spanning elementary subgraphs of $B$.
	\end{proof}
	
	\begin{lem}\label{2ind}
		Let $B\in\B(\infty)$ be a digraph that has a perfect matching. Let $\Gamma$ and $\Gamma'$ be the cycles in $B$. If both $\Gamma$ and $\Gamma'$ are independent, then $$\det A(B)=\left\{\begin{array}{ll}
			4\times (-1)^{\frac{n}{2}}\left[1-w(\Gamma)(-1)^\frac{|\Gamma|}{2}\right] & \text{ if } w(\Gamma)=\pm 1 \text{ and }w(\Gamma')=\pm\i\\
			4\times (-1)^\frac{n}{2} & \text{ if } w(\Gamma),w(\Gamma')\in\{\pm\i\}.
		\end{array}\right.$$
	\end{lem}
	\begin{proof}		
		The $\det A(B)$ follows from Lemma~\ref{9span} and Remark~\ref{det}.
	\end{proof}

	The following theorem is the major result of this subsection which provides complete characterization of non-singular digraphs in $\B(\infty)$.
	
	\begin{thm}\label{thmi}
		Let $B \in \Bi$ be a digraph, and let $\Gamma$ and $\Gamma'$ be cycles in $B$ such that $w(\Gamma') = \pm \i$. Then $B$ is non-singular if and only if one of the following conditions holds:
		\begin{enumerate}[1.]
			\item $B$ has a unique perfect matching.
			\item $B$ has more than one perfect matching, and either $\Gamma$ is not independent and $\Gamma'$ is independent, or $\Gamma$ is an independent cycle in $B$ satisfying one of the following:
			\begin{enumerate}[i)]
				\item $w(\Gamma)=\pm\i$.
				\item $w(\Gamma)=1$ and $|\Gamma|=4k+2$.
				\item $w(\Gamma)=-1$ and $|\Gamma|=4k$.
			\end{enumerate}
			\item $B$ has no perfect matching and contains an odd independent cycle $\Gamma$ with $w(\Gamma)=\pm1$.
		\end{enumerate}
	\end{thm}
	\begin{proof}
		The theorem follows from Corollary~\ref{upm}, Lemma~\ref{noper}, Lemma~\ref{1ind}, and Lemma~\ref{2ind}.
	\end{proof}
	
	\subsubsection{$\theta-$ Type bicyclic digraph}
	
	We continue our investigation of non-singular digraphs and focus on identifying non-singular $\theta-$type 3-colored bicyclic digraphs.
	
	\begin{lem}\label{lem}
		Let $B \in \B(\theta)$ be a digraph with more than one perfect matching. Then either exactly one cycle of $B$ is independent, or all three cycles are independent. 
	\end{lem}
	\begin{proof}
		By Proposition~\ref{mperf}, at least one cycle in $B$ is independent. If $B$ contains no other independent cycle, then the result follows. Hence, assume that $B$ has two independent cycles, $\Gamma_1$ and $\Gamma_2$. By Proposition~\ref{mperf}, $B-\Gamma_1$ and $B-\Gamma_2$ each has a perfect matching. Let $\M_1$ be a perfect matching of $B-\Gamma_1$, and let $\M_2$ be a perfect matching of $B-\Gamma_2$. Then $\M=\M_1 \cap \M_2$ is a perfect matching of $B-(\Gamma_1 \cup \Gamma_2)$, which implies that $|V(B-(\Gamma_1 \cup \Gamma_2))|$ is even. Consequently, $|V(\Gamma_1 \cup \Gamma_2)|$ is also even, since both $|V(B)|$ and $|V(B-(\Gamma_1 \cup \Gamma_2))|$ are even and $|V(B)|=|V(\Gamma_1\cup\Gamma_2)|+|V(B-(\Gamma_1\cup\Gamma_2))|$. Since
		\[
		|V(\Gamma_1 \cup \Gamma_2)| = |V(\Gamma_1)| + |V(\Gamma_2)| - |V(\Gamma_1 \cap \Gamma_2)|,
		\]
		it follows that $|V(\Gamma_1 \cap \Gamma_2)| = |V(\mathcal{P}_1)|$ is even. Let $\mathcal{P}_1 = [i_1, \dots, i_m]$. Then
		\[
		\M \cup \{[i_2, i_3], \dots, [i_{m-2}, i_{m-1}]\}
		\]
		is a perfect matching of $B-\Gamma_3$, and hence $\Gamma_3$ is independent in $B$.
		
	\end{proof}
	
	\begin{rem}\label{rem2}
		From Lemma~\ref{lem}, it follows that $|V(\mathcal{P}_1)|$ is even when all the cycles in $B\in\B(\theta)$ are independent. Since all these cycles are even, $|V(\mathcal{P}_2)|$ and $|V(\mathcal{P}_3)|$ are also even.
	\end{rem}

	\begin{lem}\label{t5pan}
		Let $B\in\B(\theta)$ be a digraph with a perfect matching. If all the cycles in $B$ are independent, then $B$ has six spanning elementary subgraphs: three perfect matchings and one spanning elementary subgraph containing a cycle for each cycle in $B$.
	\end{lem}
	\begin{proof}
		Let $\Gamma_1$, $\Gamma_2$, and $\Gamma_3$ be the cycles in $B$. Since all the cycles in $B$ are independent, $B$ has a spanning elementary subgraph $H_i$ containing $\Gamma_i$ for each $i=1,2,3$. To count the number of perfect matchings of $B$, let $\mathcal{P}_1 = [i, i_1, \dots, i_m, j]$, $\mathcal{P}_2 = [i, j_1, \dots, j_n, j]$, and $\mathcal{P}_3 = [i, k_1, \dots, k_p, j]$, where $m$, $n$, and $p$ are all even by Remark~\ref{rem2}. 
		
		Since all the cycles are independent, $B - (\Gamma_1 \cup \Gamma_2 \cup \Gamma_3) = B - (\mathcal{P}_1 \cup \mathcal{P}_2 \cup \mathcal{P}_3)$ has a unique perfect matching $\M$. It is clear that the following are three distinct perfect matchings of $B$:
		
		\noindent
		$\M \cup \{[i, i_1], \dots, [i_m, j]\} \cup \{[j_1, j_2], \dots, [j_{n-1}, j_n]\} \cup \{[k_1, k_2], \dots, [k_{p-1}, k_p]\}$,\\
		$\M \cup \{[i_1, i_2], \dots, [i_{m-1}, i_m]\} \cup \{[i, j_1], \dots, [j_n, i]\} \cup \{[k_1, k_2], \dots, [k_{p-1}, k_p]\}$, and \\
		$\M \cup \{[i_1, i_2], \dots, [i_{m-1}, i_m]\} \cup \{[j_1, j_2], \dots, [j_{n-1}, j_n]\} \cup \{[i, k_1], \dots, [k_p, j]\}$. Hence, $B$ has exactly six spanning elementary subgraphs as claimed.
	\end{proof}
	
	\begin{lem}\label{thetaperf2}
		Let $B\in\Bt$ be a digraph with a perfect matching. If all the cycles in $B$ are independent, then $B$ is non-singular.
	\end{lem}
	\begin{proof}
		Let $\Gamma_1$, $\Gamma_2$, and $\Gamma_3$ be the cycles in $B$ such that $w(\Gamma_1)=\pm1$. 
%
%
		By Lemma~\ref{t5pan}, $B$ has three perfect matchings and three spanning elementary subgraphs containing a cycle for each cycle.
		Since $w(\Gamma_1)=\pm1$ and weights of the other two cycles are $\pm\i$, using Remark~\ref{det}, we get
		\begin{equation*}
			\begin{split}
				\det A(B)&=3\times (-1)^{\frac{n}{2}}-2w(\Gamma_1)\times (-1)^{\frac{n+|\Gamma_1|}{2}}\\
					&=(-1)^\frac{n}{2}\left[3-2w(\Gamma_1)\times(-1)^{\frac{|\Gamma_1|}{2}}\right]\\
					&\neq 0.
			\end{split}
		\end{equation*}
	\end{proof}
	
	We summarize the characterization of non-singular digraphs in $\B(\theta)$ in the following theorem.
	
	\begin{thm}\label{thmt}
		Let $B\in\Bt$ be a digraph, and let $\Gamma_1$, $\Gamma_2$, and $\Gamma_3$ be cycles in $B$. Then $B$ is non-singular if and only if one of the following holds:
		\begin{enumerate}[1.]
			\item $B$ has a unique perfect matching.
			\item $B$ has more than one perfect matching and satisfies one of the following:
			\begin{enumerate}[i)]
				\item $B$ has more than one independent cycles.
				\item $B$ has exactly one independent cycle $\Gamma_1$ such that one of the following holds:\begin{enumerate}[a)]
					\item $w(\Gamma_1)=1$ and $|\Gamma_1|=4k+2$
					\item $w(\Gamma_1)=-1$ and $|\Gamma_1|=4k$
					\item $w(\Gamma_1)=\pm\i$.
				\end{enumerate}
			\end{enumerate}
			\item $B$ has no prefect matching and contains an odd independent cycle $\Gamma_1$ with $w(\Gamma_1)=\pm1$.
		\end{enumerate}
	\end{thm}
	\begin{proof}
		The result follows from Corollary~\ref{upm}, Lemma~\ref{noper}, Lemma~\ref{1ind}, and Lemma~\ref{thetaperf2}.
	\end{proof}

	\section{3-colored digraphs that has zero diagonal in the inverse of the adjacency matrix}\label{3}
	
	In this section, we identify all non-singular digraphs in $\mathscr{U}$ and $\B$ such that the diagonal entries of $A(B)^{-1}$ are all $zero$. The following theorem from \cite{DK2} is crucial for our analysis.
	
	\begin{lem}(\cite{DK2}, Theorem 3.1)\label{inventry}
		Let $G$ be a non-singular 3-colored digraph and $A(G)$ be its adjacency matrix. Let $A(G)^{-1}=[b_{ij}]$, then 
		\begin{equation*}
			b_{ij}=\left\{\begin{array}{ll}
				\frac{1}{\det A(G)}\sum_{P\in \mathscr{P}_{ij}}\left[w(P)\left(\sum_H(-1)^{|V(G)|-|S_H|+|P|}2^{|C_H|}\right)\right] & \textit{ if }i\neq j\\
				\frac{1}{\det A(G)}\det A(G)_{ii} & \text{otherwise},
			\end{array}\right.
		\end{equation*}
		where $\mathscr{P}_{ij}$ is the set of all $i\leadsto j$ paths $P$ such that $G-P$ has a contributing spanning elementary subgraph, $H$ is a contributing spanning elementary subgraph of $G-P$, and $A(G)_{ii}$ is the matrix obtained by deleting the $i^{th}$ row and the $i^{th}$ column of $A(G)$.
	\end{lem}

	Note that if $V(G)=V(P)$, then $G-P$ has a contributing spanning elementary subgraph vacuously.
	
	\begin{rem}\label{rem3}
		Let $G$ be a non-singular 3-colored digraph and $A(G)^{-1}=[b_{ij}]$. From Lemma~\ref{inventry}, $b_{ii}=0$ if and only if $\det A(G_i)=0$, where $G_i=G-i$.
	\end{rem}

	For digraphs in $\mathscr{U}$ we immediately obtain the following result. Although the same result was established in \cite{DK2} for digraphs in \(\mathscr{U}\) having a unique perfect matching, we observe that the uniqueness of the perfect matching is not necessary.
	
	\begin{thm}\label{unibii}
		Let $U\in\mathscr{U}$ be non-singular, and let $A(U)^{-1}=[b_{ij}]$. Then $b_{ii}=0$ for all $i\in V(U)$.
	\end{thm}
	\begin{proof}
		Let $i\in V(U)$ be arbitrary, and let $U_i = U - i$. By Theorem~\ref{unons}, $U$ has a perfect matching. Consequently, $|V(U_i)|$ is odd and thus $U_i$ does not have a perfect matching. Moreover, since the unique cycle in $U$ has weight $\pm\i$, it follows from Corollary~\ref{npm} that $\det A(U_i) = 0$. Hence, $b_{ii} = 0$.
	\end{proof}
	
	\subsection{Non-singular digraphs in $\B$ with zero diagonal in the inverse of the adjacency matrix}
		
	We now turn to the characterization of non-singular digraphs $B \in \B$ for which $A(B)^{-1}$ has zero diagonal.

	\begin{lem}\label{lem2.1}
		Let $G$ be a non-singular 3-colored digraph with a perfect matching and let $A(G)^{-1}=[b_{ij}]$. If every cycle in $G$ is even, then $b_{ii}=0$ for all $i\in V(G)$.
	\end{lem}
	\begin{proof}
		Let $i \in V(G)$, and let $G_i = G - i$. Since $G$ has a perfect matching, $G_i$ does not admit any perfect matching. Moreover, as $G$ contains no odd cycle, it follows from Lemma~\ref{evennomatch} that $\det A(G_i) = 0$ for all $i \in V(G)$. Hence, by Remark~\ref{rem3}, we have $b_{ii} = 0$ for all $i \in V(G)$.
	\end{proof}

	\begin{lem}\label{lem2.3}
		Let $G$ be a non-singular 3-colored digraph with a perfect matching and let $A(G)^{-1}=[b_{ij}]$. If $G$ does not contain an odd cycle of weight $\pm1$, then $b_{ii}=0$ for all $i\in V(G)$.
	\end{lem}
	\begin{proof}
		Let $i \in V(G)$, and let $G_i = G - \{i\}$. Since $|V(G_i)|$ is odd, any spanning elementary subgraph of $G_i$ must contain an odd cycle. As the weight of every odd cycle in $G$ is $\pm\i$, the graph $G_i$ admits no contributing spanning elementary subgraph. Therefore, by Lemma~\ref{adjdet}, $\det A(G_i) = 0$. Hence, by Remark~\ref{rem3}, we have $b_{ii} = 0$ for all $i \in V(G)$.
	\end{proof}

	For further development, we introduce the following two definitions.
	
	
	\begin{defn}
		Let \(G\) be a graph and let \(\mathcal M\) be a perfect matching of \(G\). Let \(\Gamma\) be a cycle in \(G\). An edge \([v,v']\in\mathcal M\) is called a \textbf{peg} of \(\Gamma\) relative to \(\mathcal M\) if either:\\
		1. exactly one of \(v\) and \(v'\) belongs to \(\Gamma\); or\\
		2. both \(v\) and \(v'\) belong to \(\Gamma\), and \([v,v']\) is a chord of \(\Gamma\).
	
		Consequently, every odd cycle has an odd number of pegs. Moreover, an even cycle has no pegs if it is independent in \(G\), whereas it has an even number of pegs if it is not independent, provided that each peg which is a chord of \(\Gamma\) is counted twice.
	\end{defn}

	\begin{defn}
		Let $G$ be a (di)graph with a perfect matching $\M$. A path $P=[i_1,i_2,\dots,i_n]$ in $G$ is called as \textbf{mm-alternating} (relative to the perfect matching $\M$) if the edges of $P$ alternately belongs to $\M$, with the condition  that $[i_1,i_2],[i_{n-1},i_n]\in \M$.
	\end{defn}
	
	\begin{rem}\label{rem4}
		Let $B\in\B$ be a graph possessing a perfect matching $\M$. Let $P$ be a path in $B$ such that $\Gamma$ is independent in $B-P$, and $H$ be a spanning elementary subgraph of $B-P$ containing $\Gamma$. If $\Gamma$ has a peg $[v,v']$ ($v\in V(\Gamma)$) relative to $\M$, then there exists a $v\leadsto i_v'$ \al path $[v,v',i_1,i_1',\ldots,i_v,i_v']$ such that $i_v'\in V(P)\cup V(\Gamma)$. 
		
		To see this, observe that if $v'\in V(P)$ or $v'\in V(\Gamma)$, we may set $i_v'=v'$. Otherwise, there exists a vertex $i_1$ such that $[v',i_1]$ is an independent edge in $H$. Since $\M$ is a perfect matching, there is a vertex $i_1'$ such that $[i_1,i_1']\in \M$. If $i_1'\in V(P)$ or $i_1'\in V(\Gamma)$, we set $i_v'=i_1'$. Otherwise, we repeat the process until we find such a vertex $i_v'$.
	\end{rem}

	The following lemma gives a condition under which a cycle can be a component of a spanning elementary subgraph of $B-P$ for a path $P$.
	
	\begin{lem}\label{span}
		Let $B\in\B$ be a digraph with a perfect matching, and let $\Gamma$ be a cycle in $B$. If $\Gamma$ has at least three pegs relative to every perfect matching of $B$, then $\Gamma$ is not independent in $B-P$ for any path $P$.
	\end{lem}
	\begin{proof}
		Let $\M$ be a perfect matching of $B$. First, assume that $\Gamma$ has exactly three pegs relative to $\M$. Let $[u,u']$, $[v,v']$, and $[w,w']$ be the pegs on $\Gamma$, where $u,v,w \in V(\Gamma)$. Define
		$
		\M_0=\M\setminus\bigl(E(\Gamma)\cup \{[u,u'],[v,v'],[w,w']\}\bigr),
		$
		so that $\M_0$ is a perfect matching of $B-(\Gamma\cup \{u',v',w'\})$.
		
		To obtain a contradiction, assume that $\Gamma$ is independent in $B-P$ for some path $P$. By Remark~\ref{rem4}, there exist vertices $i_u'$, $j_v'$, and $k_w'$ and corresponding \al paths $[u,u',\dots,i_u,i_u']$, $[v,v',\dots,j_v,j_v']$, and $[w,w',\dots,k_w,k_w']$ such that $i_u',j_v',k_w'\in V(P)\cup V(\Gamma)$.
		
		Observe that at most one of these vertices can lie in $V(P)$. Indeed, if this were not the case—say $i_u',j_v'\in V(P)$ and $k_w'\in V(\Gamma)$—then the cycles $\Gamma,\quad [w,\dots,k_w,k_w',\dots,w]$,\newline $ [w,\dots,k_w,k_w',\dots,u,\dots,w]$, and $[u,u',\dots,i_u',\dots,j_v',\dots,v',v,\dots,u]$ would all be distinct cycles in $B$, which is impossible since $B$ has at most three cycles (see Figure~\ref{ne}(a)).
		
		Next, assume that $i_u'\in V(P)$, so that $j_v',k_w'\in V(\Gamma)$. We must have $j_v=w'$ and $j_v'=w$; otherwise, by an argument analogous to the one above, $B$ would contain more than three distinct cycles. In this case, the cycles $\Gamma$, $[v,v',\dots,w',w,\dots,v]$, and $[v,v',\dots,w',w,\dots,u,\dots,v]$ are three distinct cycles in $B$ (see Figure~\ref{ne}(b)).
		
		Label the vertices of $\Gamma$ as
		\[
		\Gamma=[u,p_1,p_1',\dots,p_u,p_u',v,q_1,\dots,q_v',w,r_1,\dots,r_w',u].
		\]
		Define $\M'$ to be the union of $\M_0$ with the edges in
		$
		\{[p_1,p_1'],\dots,[p_u,p_u'],[r_w',r_w],\dots,[r_1',r_1]\},$
		$\{[v,q_1],[q_1',q_2],\dots,[q_v',w]\},$
		$\{[v',j_1],\dots,[j_{v-1}',w']\},$ and $\{[u,u']\}$. Then $\M'$ is a perfect matching of $B$, relative to which $[v,v']$ and $[w,w']$ are not pegs on $\Gamma$.
		
		Since $\Gamma$ has at least three pegs relative to every perfect matching of $B$, there must exist pegs $[x,x']$ and $[y,y']$ on $\Gamma$ relative to $\M'$. By Remark~\ref{rem4}, there exist vertices $i_x'$ and $i_y'$ such that $i_x',i_y'\in V(\Gamma)\cup V(P)$. Since at most one of $i_u'$, $i_v'$, $i_w'$, $i_x'$, and $i_y'$ can be in $V(P)$ and $i_u'\in V(P)$, we mist have $i_x',i_y'\in V(\Gamma)$. This again leads to a contradiction, since $[x,x',\dots,i_x',\dots,i_y',\dots,y',y,\dots,x]$ would form another cycle in $B$.
		
		A similar contradiction arises when all of $i_u',j_v',k_w'$ lie in $V(\Gamma)$. The argument extends analogously when $\Gamma$ has more than three pegs relative to $\M$. Hence, $\Gamma$ cannot be independent in $B-P$.
		
		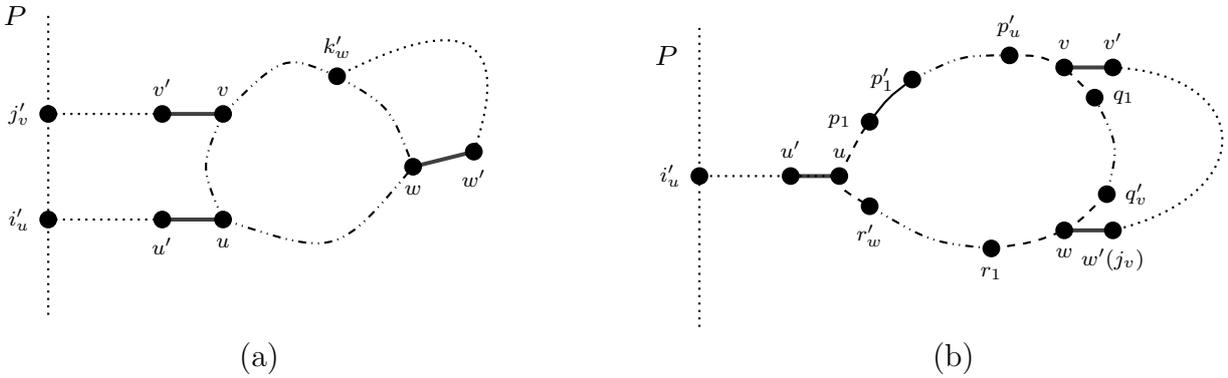
\begin{figure}[H]
			\begin{center}
				\begin{tabular*}{\textwidth}{@{\extracolsep{\fill}} cc}
					\begin{tikzpicture}
						\SetVertexStyle[MinSize=0.1,FillColor=black]
						\Vertex[y=0.7,label=$v$,position=90]{v}
						\Vertex[y=-0.7,label=$u$,position=-90]{u}
						\Vertex[x=-0.8,y=0.7,label=$v'$,position=90]{v'}
						\Vertex[x=-0.8,y=-0.7,label=$u'$,position=-90]{u'}
						\Edge(u)(u')
						\Edge(v)(v')
						\Vertex[x=-2.3,y=0.7,label=$j_v'$,position=180]{jv}
						\Vertex[x=-2.3,y=-0.7,label=$i_u'$,position=180]{iu}
						\Vertex[x=-2.3,y=-2,Pseudo]{a}
						\Vertex[x=-2.3,y=2,label=$P$,fontscale=1.3,position=180,Pseudo]{b}
						\draw[dotted,thick] (-2.3,2)--(-2.3,-2);
						\draw[dotted,thick] (-0.8,0.7)--(-2.3,0.7);
						\draw[dotted,thick] (-0.8,-0.7)--(-2.3,-0.7);
						\draw[dash dot dot,thick] (0,-0.7)..controls(-0.3,0)..(0,0.7);
						
						\Vertex[x=2.5,y=0,label=$w$,position=-90]{w}
						\Vertex[x=3.3,y=0.2,label=$w'$,position=-90]{w'}
						\Vertex[x=1.5,y=1.2,label=$k_w'$,position=90]{iw}
						
						\draw[dash dot dot,thick] (0,0.7)..controls(0.8,1.5)..(1.5,1.2);
						\draw[dash dot dot,thick] (0,-0.7)..controls(1.5,-1.2)..(2.5,0);
						\draw[dash dot dot,thick] (1.5,1.2)..controls(2.2,0.8)..(2.5,0);
						\Edge(w)(w')
						\draw[dotted,thick] (3.3,0.2)..controls(3.5,0.8) and (4,2.5)..(1.5,1.2);
					\end{tikzpicture} & \begin{tikzpicture}[scale=0.8]
						\SetVertexStyle[MinSize=0.1,FillColor=black]
						\Vertex[label=$u$,position=90]{u}
						\Vertex[x=0.5,y=0.9,label=$p_1$,position=180]{p1}
						\Vertex[x=1.2,y=1.6,label=$p_1'$,position=180]{p1'}
						\Vertex[x=2.8,y=2,label=$p_u'$,position=90]{pu}
						\Vertex[x=3.7,y=1.8,label=$v$,position=90]{v}
						\Vertex[x=4.2,y=1.3,label=$q_1$,position=0]{q1}
						\Vertex[x=4.4,y=-0.3,label=$q_v'$,position=0]{qv}
						\Vertex[x=3.7,y=-0.9,label=$w$,position=-90]{w}
						\Vertex[x=2.5,y=-1.2,label=$r_1$,position=-90]{r1}
						\Vertex[x=0.5,y=-0.5,label=$r_w'$,position=-90]{rw}
						
						\draw[dashed,thick] (0,0)..controls(0.35,0.7)..(0.5,0.9);
						\draw[solid,thick] (0.5,0.9)..controls(0.8,1.3)..(1.2,1.6);
						\draw[dash dot dot,thick] (1.2,1.6)..controls(2,2)..(2.8,2);
						\draw[dashed,thick] (0,0)..controls(0,-0.2)..(0.5,-0.5);
						\draw[dash dot dot,thick] (0.5,-0.5)..controls(1.5,-1.1)..(2.5,-1.2);
						\draw[dashed,thick] (2.5,-1.2)..controls(3.2,-1.1)..(3.7,-0.9);
						\draw[dashed,thick] (3.7,-0.9)..controls(3.9,-0.8)..(4.4,-0.3);
						\draw[dash dot dot,thick] (4.4,-0.3)..controls(4.6,0.5)..(4.2,1.3);
						\draw[dashed,thick] (4.2,1.3)..controls(4,1.5)..(3.7,1.8);
						\draw[dashed,thick] (2.8,2)..controls(3.3,2)..(3.7,1.8);
						
						\Vertex[x=4.5,y=1.8,label=$v'$,position=90]{v'}
						\Edge(v)(v')
						\Vertex[x=4.5,y=-0.9,label=$w'(j_v)$,position=-90]{w'}
						\Edge(w)(w')
						\draw[dotted,thick] (4.5,1.8)..controls(6.8,1.7)and(7,-0.3)..(4.5,-0.9);
						\Vertex[x=-0.8,label=$u'$,position=90]{u'}
						\Edge(u)(u')
						\Vertex[x=-2.3,label=$i_u'$,position=180]{iu}
						\draw[dotted,thick] (0,0)--(-2.3,0);
						\draw[dotted,thick] (-2.3,-2.5)--(-2.3,2.5);
						\Vertex[x=-2.3,y=2,label=$P$,fontscale=1.3,Pseudo,position=-180]{n}
					\end{tikzpicture}\\
					(a) & (b)
				\end{tabular*}
			\end{center}\caption{Possible structures corresponding to the cases $i_u',j_v'\in V(P)$ and only $i_u'\in V(P)$.}\label{ne}
		\end{figure}
	\end{proof}
	
%
%

	\begin{lem}\label{lem2.4}
		Let $B \in \B$ be a non-singular digraph with a perfect matching, and let $A(B)^{-1} = [b_{ij}]$. Suppose $\Gamma$ is an odd cycle in $B$ such that $w(\Gamma) = \pm 1$. Then $b_{ii} = 0$ for all $i \in V(B)$ if and only if $\Gamma$ has at least two pegs relative to every perfect matching of $B$.
	\end{lem}
	\begin{proof}
		Let $B_i = B - i$ for a vertex $i$. Let $\Gamma$ be an odd cycle in $B$ with $w(\Gamma) = \pm 1$.
		
		First, assume that $\Gamma$ has at least two pegs relative to every perfect matching of $B$. Since $\Gamma$ is an odd cycle, it must contain at least three pegs. By Lemma~\ref{span}, $\Gamma$ is not independent in $B_i$ for any $i \in V(B)$. Moreover, $B_i$ has no other contributing spanning elementary subgraph, as $|V(B_i)|$ is odd and the weights of all cycles distinct from $\Gamma$ are $\pm \i$. Hence, by Remark~\ref{rem3}, $b_{ii} = 0$ for all $i \in V(B)$.
		
		Conversely, assume that $b_{ii} = 0$ for all $i \in V(B)$. Suppose, for the sake of contradiction, that $B$ admits a perfect matching $\M$ relative to which $\Gamma$ has exactly one peg $[v, v']$ with $v \in V(\Gamma)$. Then $\Gamma$ together with the edges $\M \setminus (E(\Gamma) \cup [v,v'])$ constitutes the unique contributing spanning elementary subgraph in $B_{v'}$. Consequently, $\det A(B_{v'}) \neq 0$, which by Remark~\ref{rem3} implies $b_{v'v'} \neq 0$—a contradiction.
	\end{proof}

	For non-singular digraphs in $\B$ that do not have a perfect matching, we have the following lemma.
	
	\begin{lem}\label{lem2.2}
		Let $B \in \B$ be a non-singular digraph without a perfect matching, and let $A(B)^{-1} = [b_{ij}]$. Then there exists at least one vertex $u \in V(B)$ such that $b_{uu} \neq 0$.
	\end{lem}
	\begin{proof}
		Since $B$ is non-singular and has no perfect matching, it follows from Lemma~\ref{noper} that $B$ contains an odd independent cycle $\Gamma$ with $w(\Gamma)=\pm1$. Let $\Gamma=[1,\dots,k,1]$, and let $\M$ be a perfect matching of $B-\Gamma$. For any $u\in V(\Gamma)$, $\Gamma-u$ has a perfect matching, which together with $\M$ forms a perfect matching of $B_{u}=B-u$. 
		
		Furthermore, $B_u$ does not contain any contributing spanning elementary subgraph that includes a cycle, since the weights of all cycles other than $\Gamma$ are $\pm\i$. Therefore, by Remark~\ref{detn}, we have $\det A(B_u)=m_0\times (-1)^{\frac{|V(B_u)|}{2}}\neq0$, where $m_0$ is the number of perfect matchings of $B-\Gamma$.
	\end{proof}
	
	The following theorem is the major result of this section.
	
	\begin{thm}\label{invertible}
		Let $B\in\B$ be non-singular, and let $A(B)=[b_{ij}]$. Then $b_{ii}=0$ for all $i\in V(B)$ if and only if $B$ has a perfect matching and $B$ does not contain an odd cycle of weight $\pm1$ that has exactly one peg relative to some perfect matching of $B$.
	\end{thm}
	\begin{proof}
		The result follows from Lemma~\ref{lem2.1}, Lemma~\ref{lem2.3}, Lemma~\ref{lem2.4}, and Lemma~\ref{lem2.2}.
	\end{proof}
	
	\section{Inverse of digraphs}\label{4}
	
	In this section, we study digraph invertibility using the results developed in the preceding sections. In particular, we aim to identify digraphs in $\mathscr{U}$ and $\mathscr{B}$ whose inverses are also 3-colored digraphs.
	
	\begin{defn}
		A non-singular digraph $G$ is said to be \textbf{invertible} if the diagonal entries of $A(G)^{-1}$ are all zeros. The \textbf{inverse} of an invertible digraph $G$ is the digraph $G^+$ whose adjacency matrix is $A(G)^{-1}$.
	\end{defn}

%

	It is natural to ask: \textit{Is the inverse of an invertible 3-colored digraph necessarily a 3-colored digraph?} The answer is, in general, \textit{no}. Indeed, the inverse of an invertible digraph $G$ with $\det A(G) \neq \pm 1$ may fail to be a 3-colored digraph. Consequently, to proceed with the characterization, it is essential to identify all invertible digraphs in $\mathscr{U}$ and $\B$ whose determinant equals $\pm 1$.
	
	\begin{defn}
		A graph (digraph) $G$ is said to be \textbf{unimodular} if $\det A(G)=\pm1$.
	\end{defn}

	The following two theorems characterize unimodular digraphs in \(\mathscr{U}\) and \(\mathscr{B}\), respectively. The result for \(\mathscr{U}\) follows readily from Lemma~\ref{adjdet} (Lemma~3 of \cite{DK1}).
	
	\begin{thm}\label{uniuni}
		Let $U\in\mathscr{U}$. Then $U$ is unimodular if and only if $U$ has a unique perfect matching.
	\end{thm}
	\begin{proof}
		The result follows directly from the fact that $\det A(U)=m_0\times (-1)^{\frac{|V(U)|}{2}}$, where $m_0=\text{number of perfect matchings of $U$}$.
	\end{proof}

%

	\begin{thm}\label{uni}
		Let $B\in\B$. Then $B$ is unimodular if and only if one of the following conditions holds:
		\begin{enumerate}
			\item $B$ has a unique perfect matching.
			\item $B\in\B(\theta)$ has a perfect matching such that every cycle in $B$ is independent, and there exists a cycle $\Gamma$ satisfying $w(\Gamma)(-1)^{\frac{|\Gamma|}{2}}=1$.
		\end{enumerate}
	\end{thm}
	\begin{proof}
		First, assume that $B$ has a unique perfect matching. By Remark~\ref{rem1}, $B$ contains no independent cycle, and consequently, it has no spanning elementary subgraph containing exactly one cycle. Therefore, by Remark~\ref{det}, we obtain $\det A(B)=(-1)^{\frac{|V(B)|}{2}}$. The second case follows directly from Lemma~\ref{thetaperf2}.
		
		The converse is a direct consequence of Lemma \ref{noper}, Lemma \ref{1ind}, Lemma \ref{2ind}, and Lemma \ref{thetaperf2}.
	\end{proof}

	Next lemma is crucial for further investigations.
	
	\begin{lem}\label{uniqe=3c}
		Let $G$ be an invertible 3-colored digraph. If $G$ contains exactly one independent even cycle and the weight of this cycle is $\pm\i$, then $G^+$ is not a 3-colored digraph.
	\end{lem}
	\begin{proof}
		Let $\Gamma = [1,2,\dots,k,1]$ be the unique independent even cycle in $G$. By Lemma~\ref{mperf}, $G - \Gamma$ has a perfect matching $\M$. Consequently, \(G\) has a perfect matching. Moreover, by Lemma~\ref{mperf}, \(G\) cannot contain an independent cycle of odd length. Since \(\Gamma\) is the unique independent even cycle in \(G\), it follows that \(\Gamma\) is the unique independent cycle in \(G\). As a result, $\M$ is the unique perfect matching of $G - \Gamma$. 
		
		Let $A(G)^{-1} = [b_{ij}]$. We show that $b_{12} \notin \{\pm 1, \pm \i\}$.
		
		Observe that there are exactly two \(1\leadsto 2\) paths, namely \(P=[1,2]\) and \(Q=[1,k,\dots,2]\), such that both \(G-P\) and \(G-Q\) have a perfect matching. Indeed, suppose that there exists another \(1\leadsto 2\) path \(R\) such that \(G-R\) has a perfect matching. Then \(P\cup R\) would form an independent even cycle in \(G\) distinct from \(\Gamma\), contradicting the uniqueness of \(\Gamma\).
		
		Moreover, both \(G-P\) and \(G-Q\) have exactly one perfect matching, say \(\mathcal M_P\) and \(\mathcal M_Q\), respectively, since neither contains an independent cycle. Since \(w(\Gamma)=\pm\i\) and \(w(\Gamma)=w(P)w(Q)\), we have \(w(P)=\pm\i\,w(Q)\). Hence, by Lemma~\ref{inventry},
		\[
		b_{12}=\frac{1}{\det A(G)}
		\big[w(P)(-1)^{|V(G)|-|\mathcal M_P|+|P|}
		+w(Q)(-1)^{|V(G)|-|\mathcal M_Q|+|Q|}\big],
		\]
		which equals \(\pm\frac{1}{\det A(G)}(1\pm\i)\). Since $\det A(G)\in \mathbb{Z}\setminus\{0\}$, $b_{12}\notin \{\pm1,\pm\i\}$ and \(G^+\) is not a \(3\)-colored digraph.
	\end{proof}

	The following result, due to Kalita and Sarma \cite{DK2}, characterizes the unicyclic 3-coloured digraphs with a unique perfect matching whose inverses are 3-coloured digraphs.
	
	\begin{lem}(Theorem~3.10 \cite{DK2})\label{dk3}
		Let $U\in\mathscr{U}$ be invertible with a a unique perfect matching, and let $\Gamma$ be the unique cycle in $U$. Then $U^+$ is a 3-colored digraph if and only if either $U$ is non-bipartite or $\Gamma$ has at least four pegs.
	\end{lem}

	The following theorem provides a complete characterization of the digraphs \(U\in\mathscr U\) for which \(U^+\) is a \(3\)-coloured digraph, thereby extending the above result to the entire class \(\mathscr U\), without assuming that \(U\) has a unique perfect matching.	
	
	

	\begin{thm}\label{uni3in}
		Let $U \in \mathscr{U}$ be invertible, and let $\Gamma$ denote its unique cycle. Then $U^+$ is a $3$-coloured digraph if and only if $U$ has a unique perfect matching and the number of pegs on $\Gamma$ is different from two.
	\end{thm}
	
	\begin{proof}
		Let $A(U)^{-1}=[b_{ij}]$.
		
		First, suppose that $U$ has a unique perfect matching and that the number of pegs on $\Gamma$ is different from two. Since $U$ has a unique perfect matching, $\Gamma$ must have at least two pegs. Thus, by our assumption, it has at least three pegs. Moreover, since an even cycle has an even number of pegs, either $U$ is non-bipartite or $\Gamma$ has at least four pegs. It follows from Lemma~\ref{dk3} that $U^+$ is a $3$-coloured digraph.
		
		Conversely, suppose that $U^+$ is a $3$-coloured digraph. By Theorem~\ref{unons}, $U$ has a perfect matching. Furthermore, Lemma~\ref{uniqe=3c} implies that $U$ contains no independent cycle. Consequently, by Remark~\ref{rem1}, the perfect matching of $U$ is unique.
		
		It remains to show that $\Gamma$ cannot have exactly two pegs. Suppose, to the contrary, that it has precisely two pegs, say $[x,x']$ and $[y,y']$, where $x,y\in V(\Gamma)$. Then there are exactly two $x'\leadsto y'$ \al paths, say $P$ and $Q$, in $U$. Since
		\(
		w(P)w(Q)=w(\Gamma)=\pm\i,
		\)
		we have $w(P)=\pm\i\,w(Q)$. Hence, using Lemma~\ref{inventry} and the fact that $\det A(U)=\pm1$, we obtain
		\[
		\begin{aligned}
			b_{x'y'}
			&=\pm\Big(
			w(P)(-1)^{|V(U)|-|\M_P|+|P|}
			+
			w(Q)(-1)^{|V(U)|-|\M_Q|+|Q|}
			\Big)\\
			&=\pm(1\pm\i).
		\end{aligned}
		\]
		This contradicts the assumption that $U^+$ is a $3$-coloured digraph. Therefore, $\Gamma$ has a number of pegs different from two, completing the proof.
	\end{proof}
	
	Next, we proceed to the identification of $3$-colored digraphs $B \in \B$ such that the inverse $B^+$ is also a $3$-colored digraph. We first require the following lemma.

	\begin{lem}\label{mmal}
		Let $B \in \B$ be invertible. Then $B - P$ has a contributing spanning elementary subgraph if and only if $P$ is an \al path relative to some perfect matching of $B$.
	\end{lem}
	
	\begin{proof}
		Since $B$ is invertible, it possesses a perfect matching $\M$, by Lemma~\ref{lem2.2}.
		
		First, assume that $P$ is an \al path relative to $\M$. Then $\M - P$ is a perfect matching of $B - P$, which constitutes a contributing spanning elementary subgraph of $B - P$.
		
		Conversely, suppose that $B - P$ admits a contributing spanning elementary subgraph for some path $P$. Assume, to the contrary, that $P$ is not an \al path relative to any perfect matching of $B$. In this case, $B - P$ cannot possess a perfect matching; otherwise, the set difference of of a perfect matching of $B$ and a perfect matching of $B - P$ would yield a perfect matching of $P$ relative to a perfect matching of $B$.
		
		By Remark~\ref{rem1}, any independent cycle in $B - P$ must be odd, and by Lemma~\ref{span}, such a cycle can have at most two pegs. According to Theorem~\ref{invertible}, a digraph containing an odd cycle with fewer than three pegs must have weight $\pm \i$. Therefore, every spanning elementary subgraph of $B - P$ necessarily contains a cycle of weight $\pm \i$, and hence cannot be contributing—a contradiction.
		
		It follows that $P$ is an \al path relative to some perfect matching of $B$.
	\end{proof}

	
	Naturally, we may wonder whether the inverse of every unimodular 3-colored digraph is also a 3-colored digraph. The answer, once again, is \textit{no}. This is illustrated in Example \ref{exno3}.
	
	\begin{ex}\label{exno3}
		In Figure \ref{n3col}, $B\in\B$ is invertible and $\det A(B)=-1$, yet $B^+$ is not 3-colored. The weight of the arc $(4,3)$ is $\i$, while all other edges have weight $1$. The entries $A(B)^{-1}=[b_{ij}]$ are computed using Lemma~\ref{inventry}.
		
		For illustration, consider computation of $b_{36}$. For the entry $b_{36}$, observe that there are three $3\leadsto 6$ \al paths each relative to a perfect matching of $B$, namely $P_1=[3,2,5,6]$, $P_2=[3,4,1,6]$, and $P_3=[3,4,1,2,5,6]$. Note that $w(P_1)=1$, while $w(P_2)=w(P_3)=-\i$, and $|P_1|=|P_2|=3$, $|P_3|=5$. Hence, by Lemma~\ref{inventry}, \begin{equation*}
			\begin{split}
				b_{36}&=\frac{1}{\det A(B)}\Big[w(P_1)(-1)^{6-|S_{P_1}|+|P_1|}+w(P_2)(-1)^{6-|S_{P_2}|+|P_2|}+w(P_3)(-1)^{6-|S_{P_3}|+|P_3|}\Big]\\
				&=-\Big[(-1)^{6-1+3}-\i(-1)^{6-1+3}-\i(-1)^{6-0+5}\Big]\\
				&=-1.
			\end{split}
		\end{equation*} Others entries are also calculated similarly.
		\begin{figure}[H]
			\begin{center}
				$\vcenter {\hbox{\begin{tikzpicture}
							\SetVertexStyle[MinSize=0.1,FillColor=black]
							\Vertex[label=$1$,position=90]{1}
							\Vertex[y=-1.5,label=$2$,position=-90]{2}
							\Vertex[x=1.5,y=-1.5,label=$3$,position=-90]{3}
							\Vertex[x=1.5,label=$4$,position=90]{4}
							\Vertex[x=-1.5,label=$6$,position=90]{6}
							\Vertex[x=-1.5,y=-1.5,label=$5$,position=-90]{5}
							
							\Text[y=-2.3]{$B$}
							
							\Edge[](1)(2)
							\Edge[](2)(3)
							\Edge[Direct,label=$\i$,position=right](4)(3)
							\Edge[](4)(1)
							\Edge[](1)(6)
							\Edge[](2)(5)
							\Edge[](5)(6)
				\end{tikzpicture}}}$\hspace{2cm}
				$A(B)^{-1}=\left[\begin{array}{cccccc}
					0&-\i&0&1&0&\i\\
					\i&0&1&0&-\i&0\\
					0&1&0&0&0&-1\\
					1&0&0&0&-1&0\\
					0&\i&0&-1&0&1-\i\\
					-\i&0&-1&0&1+\i&0\\
				\end{array}\right]$
			\end{center}\caption{A graph $B\in\B$ such that $B^+$ is not a 3-colored digraph.}\label{n3col}
		\end{figure}
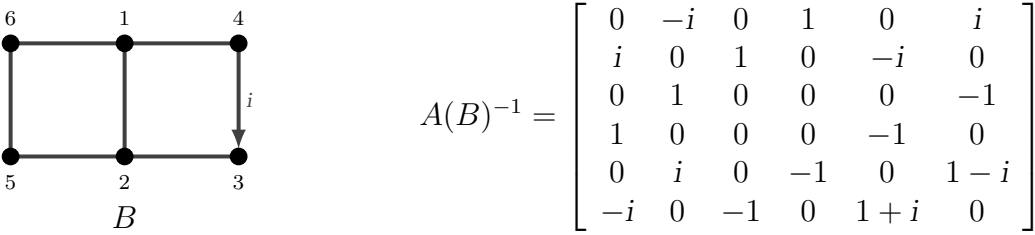
	\end{ex}
	
	From Example~\ref{exno3}, it is clear that not all unimodular $3$-colored digraphs have $3$-colored inverses. Notice that $B$ possesses more than one perfect matching. Next lemma shows that unique perfect matching is necessary for $B^+$ to be 3-colored. Let $|V(B)|=n$ for the rest of this section.
	
	
	\begin{lem}\label{unino3col}
		Let $B\in\B$ be invertible. If $B$ does not possess a unique perfect matching, then $B^+$ is not a 3-colored digraph.
	\end{lem}
	
	\begin{proof}
		Let $A(B)^{-1}=[b_{ij}]$. By Theorem~\ref{invertible}, the existence of $B^+$ implies that $B$ has a perfect matching. Since $B$ does not possess a unique perfect matching, $B$ has at least two perfect matchings and therefore, an independent even cycle.
		
		By Lemma~\ref{uniqe=3c}, if $B$ contains a unique independent even cycle of weight $\pm\i$, then $B^+$ is not a 3-colored digraph, and the proof is complete. Consequently, we may assume that either $B$ has more than one independent cycle, or $B$ has a unique independent cycle whose weight is $\pm1$. We consider these cases separately.

		\noindent\textbf{Case 1.} The unique independent cycle in $B$ has weight $\pm1$.
		
		Since $B$ is nonsingular, we have $w(\Gamma)(-1)^{\frac{|\Gamma|}{2}}=-1$ and $\det A(B)=4(-1)^{\frac{n}{2}}$, by Lemma~\ref{1ind}. Let $\Gamma=[1,2,\dots,k,1]$ denote the independent cycle, where $k$ is even. Take the paths $P=[1,2]$ and $Q=[1,k,\dots,2]$. Then both $B-P$ and $B-Q$ have unique perfect matchings, denoted by $\M_P$ and $\M_Q$, respectively.
		
		Note that $|P|=1$, $|Q|=|\Gamma|-1$, $|\M_P|=\frac{n-2}{2}$, and $|\M_Q|=\frac{n-|\Gamma|}{2}$. Moreover, since
		$
		\pm1=w(\Gamma)=w(P)w(Q),
		$
		it follows that $w(Q)=w(\Gamma)w(P)$. Applying Lemma~\ref{inventry}, we obtain
		\begin{equation*}
			\begin{split}
				b_{12}
				&=\frac{(-1)^{\frac{n}{2}}}{4}\big[w(P)(-1)^{n-|\M_P|+|P|}+w(Q)(-1)^{n-|\M_Q|+|Q|}\big]\\
				&=\frac{1}{4}\big[w(P)+w(P)w(\Gamma)(-1)^{\frac{|\Gamma|}{2}+1}\big]\\
				&=\frac{w(P)}{4}\big[1-w(\Gamma)(-1)^{\frac{|\Gamma|}{2}}\big]\\
				&=\frac{w(P)}{2},
			\end{split}
		\end{equation*}
		where we used the fact that $w(\Gamma)(-1)^{\frac{|\Gamma|}{2}}=-1$. Hence $b_{12}\notin\{0,\pm1,\pm\i\}$, and therefore $B^+$ is not a 3-colored digraph.
		
		\medskip
		\noindent\textbf{Case 2.} $B$ has more than one independent cycles.
		
		\smallskip
		\textbf{Case 2a.} $B\in\B(\infty)$ and all cycles are independent.
		
		Let $\Gamma$ and $\Gamma'$ be cycles in $B$, with at least one having weight $\pm\i$. Let $\Gamma=[1,2,\dots,k,1]$ be such that $w(\Gamma)=\pm\i$. We have exactly the two $1\leadsto2$ paths $P=[1,2]$ and $Q=[1,k,\dots,2]$.
		
		Since $\Gamma'$ is independent in $B$, both $B-P$ and $B-Q$ admit two perfect matchings and a spanning elementary subgraph containing $\Gamma'$. Denote these perfect matchings by $\M_{P1},\M_{P2}$ and $\M_{Q1},\M_{Q2}$, respectively. Then
		\[
		|\M_{P1}|=|\M_{P2}|=\frac{n-2}{2},\qquad
		|\M_{Q1}|=|\M_{Q2}|=\frac{n-|\Gamma|}{2}.
		\]
		Moreover, $w(P)=\pm\i\, w(Q)$. Applying Lemma~\ref{inventry} and simplifying the exponents yields
		\begin{equation*}
			\begin{split}
				b_{12}
				&=\frac{\pm2}{\det A(B)}\Big[\left(w(P)-w(Q)(-1)^{\frac{|\Gamma|}{2}}\right)-\mathfrak{R}(w(\Gamma'))\left(w(P)(-1)^{\frac{|\Gamma'|}{2}}-w(Q)(-1)^{\frac{|\Gamma|+|\Gamma'|}{2}}\right)\Big]\\
				&=\frac{\pm 2}{\det A(B)}\Big[\left(w(P)-w(Q)(-1)^{\frac{|\Gamma|}{2}}\right)
				\left(1-\mathfrak{R}(w(\Gamma'))(-1)^{\frac{|\Gamma'|}{2}}\right)\Big]\\
				&\notin \{0,\pm1,\pm\i\},
			\end{split}
		\end{equation*}
		since either $w(\Gamma')=\pm\i$ or $w(\Gamma')(-1)^{\frac{|\Gamma'|}{2}}=-1$ by Theorem~\ref{thmi}. Hence $B^+$ is not a 3-colored digraph.
		
		\smallskip
		\textbf{Case 2b.} $B\in\B(\theta)$ and has at least two independent cycles. By Lemma~\ref{lem}, all cycles are independent.
		
		Let $\Gamma_1,\Gamma_2,\Gamma_3$ denote the cycles in $B$, and let $\M$ be the perfect matching of $B-(\Gamma_1\cup\Gamma_2\cup\Gamma_3)$. Let $\mathcal{P}_1 = [i, u_1, \dots, u_k, j]$, $\mathcal{P}_2 = [i, v_1, \dots, v_l, j]$, and $\mathcal{P}_3 = [i, w_1, \dots, w_m, j]$ be the $i\leadsto j$ paths described in Remark~\ref{paths}, where $k,l,m$ are all even by Remark~\ref{rem2}. Without loss of generality, assume $w(\Gamma_1)=\pm1$ and $w(\Gamma_2),w(\Gamma_3)\in\{\pm\i\}$.
		
		Since $w(\Gamma_1)=\pm w(\mathcal{P}_1)w(\mathcal{P}_2)$, we have $w(\mathcal{P}_1)=\pm w(\mathcal{P}_2)$, and similarly $w(\mathcal{P}_3)=\pm\i\,w(\mathcal{P}_2)$. Consider the edge $[u_1,u_2]$. There are three $u_1\leadsto u_2$ paths:
		\[
		Q_1=[u_1,u_2],\quad
		Q_2=[u_1,i,\mathcal{P}_2,j,u_k,\dots,u_2],\quad
		Q_3=[u_1,i,\mathcal{P}_3,j,u_k,\dots,u_2].
		\]
		
		Observe that \begin{enumerate}[]
			\item $\M_1=\M\cup \{[i,v_1],\dots,[v_l,j],[w_{m},w_{m-1}],\dots,[w_2,w_1]\}\cup \{[u_3,u_4],\dots,[u_{k-1},u_k]\}$ and
			\item  $\M_2=\M\cup \{[v_1,v_2],\dots,[v_{l-1},v_l],[j,w_m],\dots,[w_1,i]\}\cup \{[u_3,u_4],\dots,[u_{k-1},u_k]\}$
		\end{enumerate} are two perfect matchings of $B - Q_1$. Moreover, $H=\M\cup \Gamma_3\cup \{[u_3,u_4],\dots,[u_{k-1},u_k]\}$ is a spanning elementary subgraph of $B-Q_1$ that contains $\Gamma_3$. Similarly, it can be seen that both $B - Q_2$ and $B - Q_3$ each admit exactly one perfect matching and no other spanning elementary subgraphs. Since $w(\Gamma_3) = \pm\i$, $H$ is not a contributing spanning elementary subgraph of $B - Q_1$. Hence, by Lemma~\ref{inventry},
		\[
		b_{u_1u_2}
		=\pm\Big[2w(Q_1)(-1)^{m_1}+w(Q_2)(-1)^{m_2}+w(Q_3)(-1)^{m_3}\Big],
		\]
		where the exponents are as given in Lemma~\ref{inventry} and  $w(Q_2) = \pm \frac{w(\mathcal{P}_1) w(\mathcal{P}_2)}{w(Q_1)}$ and $w(Q_3) = \pm \frac{w(\mathcal{P}_1) w(\mathcal{P}_3)}{w(Q_1)}$. Because one of $w(\mathcal{P}_2)$ and $w(\mathcal{P}_3)$ is real and the other imaginary, it follows that
		$b_{u_1u_2}\notin\{0,\pm1,\pm\i\}$. Thus, $B^+$ is not a 3-colored digraph.
	\end{proof}

	The following lemma shows that having a unique perfect matching is not sufficient for $B^+$ to be 3-colored. The number of pegs on the cycle also plays a role in determining whether $B^+$ is a 3-colored digraph.
	
	\begin{lem}\label{lem5.1}
		Let $B \in \B$ be an invertible digraph with a unique perfect matching. If $B$ does not contain a cycle that has exactly two pegs, then $B^+$ is a 3-colored digraph.
	\end{lem}
	
	\begin{proof}
		Let $A(B)^{-1}=[b_{ij}]$, and let $k,l\in V(B)$. If there is no $k\leadsto l$ \al path in $B$, then by Lemma~\ref{inventry} and Lemma~\ref{mmal}, we have $b_{kl}=0$. Thus, assume that there exists a $k\leadsto l$ \al path $P$ in $B$.
		
		We claim that $P$ is the unique $k\leadsto l$ \al path in $B$. Suppose, to the contrary, that there exists another such path $Q$. Then the symmetric difference $P\Delta Q$ contains an even cycle $\Gamma$ with either no pegs or exactly two pegs. Since $B$ has a unique perfect matching, $\Gamma$ must have exactly two pegs, contradicting our assumption.
		
		Therefore, $P$ is the unique $k\leadsto l$ \al path in $B$, and consequently $B-P$ has a unique perfect matching. By Remark~\ref{rem1}, $B-P$ contains no independent cycles and hence admits no contributing spanning elementary subgraph other than the perfect matching. Applying Lemma~\ref{inventry}, we obtain
		\[
		b_{kl}=\frac{1}{\det A(B)}\,w(P)\times(-1)^{n-|\M_P|+|P|}.
		\]
		Since $B$ has a unique perfect matching, Theorem~\ref{uni} implies that $\det A(B)=\pm1$. Hence $b_{kl}\in\{\pm1,\pm\i\}$. It follows that for all $i,j\in V(B)$,
		$
		b_{ij}\in\{0,\pm1,\pm\i\},
		$
		and therefore $B^+$ is a 3-colored digraph.
	\end{proof}

	\begin{lem}\label{lem4.2}
		Let $B\in\B$ be an invertible digraph with a unique perfect matching. Suppose that $B$ contains a cycle that has exactly two pegs. Then $B^+$ is a 3-colored digraph if and only if, for every cycle $\Gamma$ in $B$ with exactly two pegs $[u,u']$ and $[v,v']$ ($u,v\in V(\Gamma)$), there exists a pair of $u'\leadsto v'$ \al paths, $P$ and $Q$, satisfying one of the following conditions:
		\begin{enumerate}[i)]
			\item $|P|-|Q|\equiv 2\pmod 4$ and $w(P)=w(Q);$
			\item $|P|-|Q|\equiv 0\pmod 4$ and $w(P)=-w(Q)$.
		\end{enumerate}
	\end{lem}
	%
	%
	\begin{proof}
		Let $A(B)^{-1}=[b_{ij}]$, and let $k,l\in V(B)$. If there exists at most one $k\leadsto l$ \al path, then $b_{kl}\in\{0,\pm1,\pm\i\}$, by Lemma~\ref{mmal}. Thus, assume that there are at least two $k\leadsto l$ \al paths, $P_1$ and $P_2$. Since $B$ has a unique perfect matching, the symmetric difference $P_1\Delta P_2$ contains an even cycle with exactly two pegs. Let $\Gamma$ be the cycle with exactly two pegs $[u,u']$ and $[v,v']$, where $u,v\in V(\Gamma)$, and let $P$ and $Q$ be two $u'\leadsto v'$ \al paths.
		
		First, assume that $B$ satisfies the stated conditions. Label the vertices of $\Gamma$ as $\Gamma=[u=u_1,u_2,\dots,u_p=v,u_{p+1},\dots,u_{q},u]$, then $P=[u',u=u_1,\dots,u_{p-1},u_p=v,v']$ and $Q=[u',u=u_1,u_q,\dots,u_{p+1},u_p=v,v']$. Observe that in a bicyclic graph, there are at most four \al paths between any two vertices. Accordingly, we have the following three cases:
		
		\smallskip
		\noindent\textbf{Case 1.} There are exactly two $k\leadsto l$ \al paths (see Figure~\ref{lemfig}(a)). Since $P_1$ and $P_2$ are the only $k\leadsto l$ \al paths in $B$, one of them must contain $P$ and the other must contain $Q$. Without loss of generality, assume that $P_1$ contains $P$ and $P_2$ contains $Q$. Then $P_1=[k,k_1,\ldots,u',P,v',\ldots,l]$ and $P_2=[k,k_1,\ldots,u',Q,v',\ldots,l]$. We see that $w(P_1)=w(S)w(P)w(T)$ and $w(P_2)=w(S)w(Q)w(T)$, where $S=[k,k_1,\ldots,u']$ and $T=[v',\ldots,l]$. Consequently, $w(P_1)=w(P_2)$ when $w(P)= w(Q)$, and $w(P_1)=-w(P_2)$ when $w(P)=-w(Q)$. Moreover, $|P_1|=|S|+|P|+|T|$ and $|P_2|=|S|+|Q|+|T|$, therefore $|P_1|-|P_2|\equiv (|P|-|Q|)\pmod 4$. Hence, by Lemma~\ref{inventry},
		\begin{equation*}
			\begin{split}
				b_{kl}&=\frac{1}{\det A(B)}\left[w(P_1)\times (-1)^{n-\frac{n-|P_1|-1}{2}+|P_1|}+w(P_2)\times (-1)^{n-\frac{n-|P_2|-1}{2}+|P_2|}\right]\\
				&=\pm \left[w(P_1)(-1)^{\frac{3|P_1|+1}{2}}+w(P_2)(-1)^{\frac{3|P_2|+1}{2}}\right]\\
				&=0.
			\end{split}
		\end{equation*}
		\smallskip
		\noindent\textbf{Case 2.} There are three $k\leadsto l$ \al paths (see Figure~\ref{lemfig}(b)). Let $P_3$ be the $k\leadsto l$ \al path distinct from $P_1$ and $P_2$. Since $w(P_3)\in \{\pm1,\pm\i\}$, we have, by Lemma~\ref{inventry},
		\begin{equation*}
			\begin{split}
				b_{kl}&=\pm\left[w(P_1)\times (-1)^{\frac{3|P_1|+1}{2}}+w(P_2)\times (-1)^{\frac{3|P_2|+1}{2}}+w(P_3)\times (-1)^{\frac{3|P_3|+1}{2}}\right]\\
				&=w(P_3)\times (-1)^{\frac{3|P_3|+1}{2}}\in \{\pm1,\pm\i\}.
			\end{split}
		\end{equation*}
		\smallskip
		\noindent\textbf{Case 3.} There are four $k\leadsto l$ \al paths (see Figure~\ref{lemfig}(c)). Let $P_3$ and $P_4$ be the $k\leadsto l$ \al paths distinct from $P_1$ and $P_2$. Then the symmetric difference $P_3\Delta P_4$ contains an even cycle, distinct from $\Gamma$, with exactly two pegs. Let $\Gamma'$ be this cycle, with pegs $[x,x']$ and $[y,y']$, where $x,y\in V(\Gamma')$. Let $R$ and $S$ be two $x'\leadsto y'$ paths satisfying the conditions of the lemma. Arguing analogously to Case~1, we find that $w(P_3)=\pm w(P_4)$ and $|P_3|-|P_4|\equiv (|R|-|S|)\pmod 4$. Hence, by Lemma~\ref{inventry},
		\begin{equation*}
			\begin{split}
				b_{kl}&=\pm \sum_{i=1}^{4}w(P_i)\times (-1)^{\frac{3|P_i|+1}{2}}\\
				b_{kl}&=\pm\left[w(P_1) (-1)^{\frac{3|P_1|+1}{2}}+w(P_2) (-1)^{\frac{3|P_2|+1}{2}}+w(P_3) (-1)^{\frac{3|P_3|+1}{2}}+w(P_4) (-1)^{\frac{3|P_4|+1}{2}}\right]\\
				&=0.
			\end{split}
		\end{equation*}
		Thus, $b_{ij}\in\{0,\pm1,\pm\i\}$ for all pairs of vertices $i,j$, and hence $B^+$ is a 3-colored digraph.
		
		The converse is straightforward.
		
		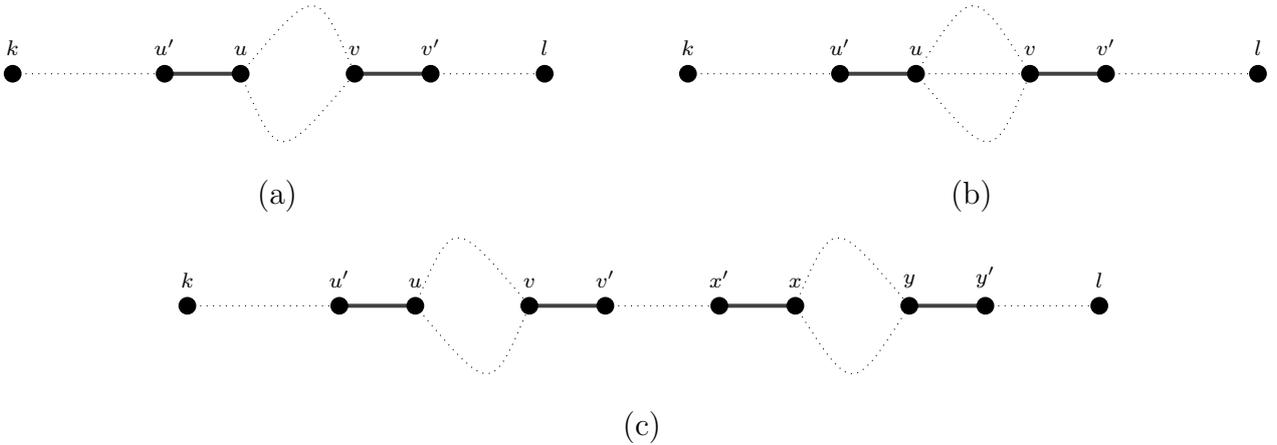
\begin{figure}[H]
			\begin{center}
				\begin{tabular*}{\textwidth}{@{\extracolsep{\fill}} cc}
					\begin{tikzpicture}
						\SetVertexStyle[MinSize=0.1,FillColor=black]
						\Vertex[label=$k$,position=90]{k}
						\Vertex[x=2,label=$u'$,position=90]{u'}
						\Vertex[x=3,label=$u$,position=90]{u}
						\Vertex[x=4.5,label=$v$,position=90]{v}
						\Vertex[x=5.5,label=$v'$,position=90]{v'}
						\Vertex[x=7,label=$l$,position=90]{l}
						\draw[dotted] (3,0)..controls(4,1.2)..(4.5,0);
						\draw[dotted] (3,0)..controls(3.5,-1.2)..(4.5,0);
						\draw[dotted] (0,0)--(2,0);
						\draw[dotted] (5.5,0)--(7,0);
						\Edge(u')(u)
						\Edge(v)(v')
					\end{tikzpicture} & \begin{tikzpicture}
						\SetVertexStyle[MinSize=0.1,FillColor=black]
						\Vertex[label=$k$,position=90]{k}
						\Vertex[x=2,label=$u'$,position=90]{u'}
						\Vertex[x=3,label=$u$,position=90]{u}
						\Vertex[x=4.5,label=$v$,position=90]{v}
						\Vertex[x=5.5,label=$v'$,position=90]{v'}
						\Vertex[x=7.5,label=$l$,position=90]{l}
						\draw[dotted] (0,0)--(2,0);
						\draw[dotted] (3,0)..controls(3.75,1.2)..(4.5,0);
						\draw[dotted] (3,0)..controls(4,-1.2)..(4.5,0);
						\draw[dotted] (3,0)--(4.5,0);
						\draw[dotted] (5.5,0)--(7.5,0);
						\Edge(u')(u)
						\Edge(v')(v)
					\end{tikzpicture}\\
					(a) & (b)\\
					\multicolumn{2}{c}{\begin{tikzpicture}
							\SetVertexStyle[MinSize=0.1,FillColor=black]
							\SetVertexStyle[MinSize=0.1,FillColor=black]
							\Vertex[label=$k$,position=90]{k}
							\Vertex[x=2,label=$u'$,position=90]{u'}
							\Vertex[x=3,label=$u$,position=90]{u}
							\Vertex[x=4.5,label=$v$,position=90]{v}
							\Vertex[x=5.5,label=$v'$,position=90]{v'}
							\Vertex[x=7,label=$x'$,position=90]{x'}
							\draw[dotted] (3,0)..controls(3.5,1.2)..(4.5,0);
							\draw[dotted] (3,0)..controls(4,-1.2)..(4.5,0);
							\draw[dotted] (0,0)--(2,0);
							\draw[dotted] (5.5,0)--(7,0);
							\Edge(u')(u)
							\Edge(v)(v')
							\Vertex[x=8,label=$x$,position=90]{x}
							\Vertex[x=9.5,label=$y$,position=90]{y}
							\Vertex[x=10.5,label=$y'$,position=90]{y'}
							\Vertex[x=12,label=$l$,position=90]{l}
							\Edge(x)(x')
							\Edge(y)(y')
							\draw[dotted] (8,0)..controls(8.5,1.2)..(9.5,0);
							\draw[dotted] (8,0)..controls(8.75,-1.2)..(9.5,0);
							\draw[dotted] (10.5,0)--(12,0);
					\end{tikzpicture}}\\
					\multicolumn{2}{c}{(c)}\\
				\end{tabular*}
			\end{center}\caption{Structures of $B$ corresponding to two, three, and four $k\leadsto l$ \al paths.}\label{lemfig}
		\end{figure}
	\end{proof}

	\begin{thm}
		Let $B\in\B$ be invertible. Then $B^+$ is a 3-colored digraph if and only if $B$ has a unique perfect matching and one of the following conditions holds:
		\begin{enumerate}
			\item $B$ contains no cycle that has exactly two pegs.
			\item $B$ contains at least one cycle that has exactly two pegs, and for every cycle that has exactly two pegs $[u,u']$ and $[v,v']$, there exists two $u'\leadsto v'$ \al paths, $P$ and $Q$, satisfying one of the following:
			\begin{enumerate}[i)]
				\item $|P|-|Q|\equiv 2\pmod 4$ and $w(P)=w(Q)$; or
				\item $|P|-|Q|\equiv 0\pmod 4$ and $w(P)=-w(Q)$.
			\end{enumerate}
		\end{enumerate}
	\end{thm}
	\begin{proof}
		The theorem follows from Lemma \ref{unino3col}, Lemma \ref{lem5.1}, and Lemma \ref{lem4.2}.
	\end{proof}
	
	\subsection*{Acknowledgments}
	
	The author also expresses gratitude to his Ph.D. supervisor, Dr. Debajit Kalita, for his guidance throughout the research.
	
	\subsection*{Declaration of competing interest}
	
	There is no competing interest.
	
	\subsection*{Funding}
	The author acknowledges the financial assistance provided by the UGC, Government of India, under the UGC SRF scheme, bearing NTA Reference No.: 191620004617.
	
	\subsection*{AI disclosure}
	ChatGPT (OpenAI, GPT-6) was used only to improve the English of the manuscript. AI/LLM tools were not used for exploration, brain-
	storming, coding, computation, mathematical reasoning, proving, or generating figures. All definitions, results, proofs, examples, and figures are entirely the work of the author,	who reviewed every language suggestion and take full responsibility for the content of the manuscript.

\end{document}